\documentclass[12pt,reqno]{amsart}
\usepackage{amsthm,amssymb,mathabx}
\usepackage{bm,fullpage}

\usepackage[
  bookmarks = true, 
  pagebackref = false, 
  pdfauthor = {Anderson Palsson}, 
  pdftitle = \text{AP}, 
  colorlinks = true, 
  linkcolor = blue, 
  citecolor = blue]
{hyperref}

\newtheorem{thm}{Theorem}

\newtheorem{lemma}{Lemma}
\newtheorem{prop}{Proposition}

\newtheorem{cor}{Corollary}

\theoremstyle{remark}

\newtheorem{remark}{Remark}

\newcommand{\R}{\ensuremath{\mathbb{R}}}
\newcommand{\T}{\ensuremath{\mathbb{T}}}
\newcommand{\Z}{\ensuremath{\mathbb{Z}}}

\newcommand{\N}{\ensuremath{\mathbb{N}}}

\newcommand{\eq}{\begin{equation}}
\newcommand{\ee}{\end{equation}}

\newcommand{\RFT}[1]{\widetilde{#1}}

\numberwithin{equation}{section}

\title{Bounds for discrete multilinear spherical maximal functions}
\author[T. Anderson]{Theresa C. Anderson}
\address{
	Department of Mathematics
	\\ Purdue University
	\\ 150 N. University St.
	\\ W. Lafayette, IN 47907
	\\	U.S.A.
}
\email{tcanderson@math.purdue.edu}

\author[E. Palsson]{Eyvindur Ari Palsson}
\address{
	Department of Mathematics
	\\	Virginia Tech 
	\\	225 Stanger St.
	\\	Blacksburg, VA 24061
	\\	U.S.A.
}
\email{palsson@vt.edu}

\begin{document}
\maketitle
\tableofcontents

\begin{abstract}
We define a discrete version of the bilinear spherical maximal function, and show bilinear $l^{p}(\Z^d)\times l^{q}(\Z^d) \to l^{r}(\Z^d)$ bounds for $d \geq 3$, $\frac{1}{p} + \frac{1}{q} \geq \frac{1}{r}$, $r>\frac{d}{d-2}$ and  $p,q\geq 1$.  Due to interpolation, the key estimate is an $l^{p}(\Z^d)\times l^{\infty}(\Z^d) \to l^{p}(\Z^d)$ bound, which holds when $d \geq 3$, $p>\frac{d}{d-2}$.  A key feature of our argument is the use of the circle method which allows us to decouple the dimension from the number of functions compared to the work of Cook.
\end{abstract}

\section{Introduction}
The study of multilinear variants of continuous operators appearing in harmonic analysis is a rich area of study.  Another active area of investigation is determining bounds for discrete operators involving integration over a curved submanifold -- these operators often exhibit radically different behaviour than their continuous counterparts.  Discrete bilinear and multilinear variants have been significantly less studied.  In this paper we combine the themes of discreteness and multilinearity with the study of the discrete bilinear spherical maximal function.  Namely, we prove $l^{p}(\Z^d)\times l^{q}(\Z^d) \to l^{r}(\Z^d)$ bounds in an open region obtained by both simple discrete $l^p$ theory and interpolation with a key estimate, which we prove.  This estimate is an $l^{p}(\Z^d)\times l^{\infty}(\Z^d) \to l^{\infty}(\Z^d)$ bound for $p>\frac{d}{d-2}$ that is obtained using the circle method from analytic number theory.  While this application introduces a number of number theoretic obstacles, it also allows us to decouple the multilinearity (number of functions) with the dimension compared to the work of Cook \cite{Cook}.  The idea for approaching an $l^{p}(\Z^d)\times l^{\infty}(\Z^d) \to l^{\infty}(\Z^d)$ estimate is classic in the continuous setting, but in the particular case of the bilinear spherical maximal function it was first employed by Barrionuevo, Grafakos, He, Honz\'{i}k and Oliveira \cite{BGHHO}.  For simplicity, we work with the bilinear version of the discrete spherical maximal function in $\Z^{2d}$ but we comment on the more general results for the $l$-linear version in the last section.

The study of spherical maximal functions dates back to Stein \cite{St76} where they naturally arose in connection with the wave equation.  This operator is bounded on $L^p(\R^d)$ for $d \geq 2$, $p >\frac{d}{d-1}$ (Stein $d \geq 3$ \cite{St76}, Bourgain $d = 2$ \cite{B86}); these ranges are sharp.  Oberlin introduced a multilinear variant with functions on $\R$, and proved bounds from $L^p(\R)\times\ldots\times L^p(\R)$ into $L^q(\R)$ for $(\frac{1}{p}, \frac{1}{q})$ lying in a polygonal region \cite{Oberlin}.  Geba, Greenleaf, Iosevich, Sawyer and the second author \cite{Geba_et_al} were the first to consider a multilinear variant with functions in $\R^d$, and proved bounds in the bilinear setting of the type $L^p(\R^d)\times L^p(\R^d) \rightarrow L^{p'}(\R^d)$ for $1\leq p \leq 2$ and $d\geq 2$.  Barrionuevo, Grafakos, He, Honz\'{i}k and Oliveira \cite{BGHHO} expanded vastly on the results by Geba et al and obtained a wide range of H\"{o}lder type estimates $L^p(\R^d)\times L^q(\R^d) \rightarrow L^r(\R^d)$ for $d\geq 8$. This was improved to $d\geq 4$ by Grafakos, He and Honz\'{i}k \cite{GHH} and the range of estimates then slightly expanded by Heo, Hong and Yang \cite{HHYpreprint}. Finally, in a recent work, Jeong and Lee proved sharp bounds for the continuous bilinear spherical maximal function in $\R^{2d}$ \cite{JL} with their method clearly extending to higher levels of multilinearity.

Magyar, Stein and Wainger considered a discrete linear spherical maximal function, first introduced by Magyar \cite{Magyar_dyadic}, and proved bounds for $d\geq 5, p> \frac{d}{d-2}$; moreover they showed that this range was sharp in $d$ and $p$.  Cook studied a version of the discrete multilinear spherical maximal function analogous to Oberlin's work and similarly proved bounds of the type $l^p(\Z)\times\ldots\times l^p(\Z)\rightarrow l^q(\Z)$ \cite{Cook}.  We continue to further the investigation of discrete multilinear spherical maximal functions by introducing the circle method technique, allowing us to consider functions on $\Z^{d}$ and obtaining a wide range of estimates.  Our range is not sharp, but approaches sharp estimates as $d \to \infty$.  We relate some necessary conditions of multilinear spherical maximal functions in the opening section as well as sharpness examples after the proof of the main theorem.  An interesting open question is to fully determine the sharp range for this operator, thus providing the discrete analogue to \cite{JL}.

We now define our bilinear discrete (or integral) spherical maximal function (we comment on the multilinear version in the last section).  The operators that we consider extend those considered in \cite{Cook}, and our boundedness results complement these as well as the continuous bounds found in \cite{Oberlin}, \cite{Geba_et_al}, \cite{BGHHO}, \cite{GHH}, \cite{HHYpreprint}, \cite{JL}.
Note that our technique is different to the approach used in \cite{Cook}, we directly import the continuous bounds as a key step. Many papers have used this technique in the linear setting, such as \cite{MSW}, \cite{Magyar_ergodic}, \cite{Hughes_Vinogradov}, and \cite{ACHK2}.

The continuous spherical averages can be written as
\[
T_\lambda(f,g)(\bm{x}) = \int_{\lambda\mathbb{S}^{2d-1}}f(\bm{x}-\bm{u})g(\bm{x}-\bm{v})d\sigma_\lambda(\bm{u},\bm{v})
\]
where $\bm{u}$ and $\bm{v}$ are vectors in $\R^d$, and $d\sigma_\lambda$ is the continuous normalized spherical measure on $\lambda\mathbb{S}^{2d-1}$.  
We can rewrite this as a convolution operator:
\[
T_\lambda(f,g)(\bm{x}) = ((f\otimes g)*d\sigma_\lambda)(\bm{x},\bm{x}).
\]
Then the maximal operator is 
\[
T^*(f,g)(\bm{x}) := \sup_{\lambda >0}|T_\lambda(f,g)(\bm{x})|.
\]

Abusing notation, the discrete version that we will consider is
\[
T^*(f,g)(\bm{x}) = \sup_{\lambda \in \N}\left| \frac{1}{N(\lambda)}\sum_{\bm{u}^2+\bm{v}^2=\lambda}f(\bm{x}-\bm{u})g(\bm{x}-\bm{v}) \right|
\]
where $\bm{u}$, $\bm{v} \in \Z^d$ and $N(\lambda) = \#\{ (\bm{u},\bm{v}) \in \Z^d\times\Z^d: \bm{u}^2+\bm{v}^2=\lambda\}$ is the number of lattice points on the sphere of radius $\lambda^{1/2}$ in $\R^{2d}$, which by the Hardy-Littlewood asymptotic is approximately $\lambda^{d-1}$ if the distribution is \emph{regular}.  Here $\bm{u}^2$ is shorthand for $u_1^2+ \dots + u_d^2$.  This operator can also be thought of as 
\[
T^*(f,g)(\bm{x}) = \sup_{\lambda >0}\left| ((f\otimes g)*\sigma_\lambda)(\bm{x},\bm{x}) \right|.
\]
where this time, $\sigma_\lambda (\bm{u},\bm{v}) =\frac{1}{N(\lambda)}\chi_{\{\bm{u}\in\Z^d, \bm{v}\in\Z^d: \bm{u}^2+\bm{v}^2 = \lambda\}}$ is the normalized arithmetic (probability) surface measure. For $N(\lambda)$ to be regular, we need $2d > 4$, or $d \geq 3$.  We will assume regularity throughout the paper.

We will prove the following:
\begin{thm}
\label{Main theorem}
$T^*$ is bounded $l^{p}(\Z^d)\times l^{\infty}(\Z^d) \to l^{p}(\Z^d)$ for all $d \geq 3$, $p >\frac{d}{d-2}$.
\end{thm}

\begin{remark}
Sharpness examples provided at the end of the proof of Theorem \ref{Main theorem} in section \ref{proof of main theorem} show that for $l^{p}(\Z^d)\times l^{\infty}(\Z^d) \to l^{p}(\Z^d)$ bounds to hold when $d\geq 3$ we must have $p>1$. As $d \to \infty$ our result approaches the sharp range.
\end{remark}

\begin{remark}
By symmetry, we also get that $T^*$ is bounded $l^{\infty}(\Z^d)\times l^{p}(\Z^d) \to l^{p}(\Z^d)$ for all $p> \frac{d}{d-2}$.  We can interpolate these bounds to get all points on the line including the $l^{2p}(\Z^d)\times l^{2p}(\Z^d) \to l^{p}(\Z^d)$ bounds for all $p> \frac{d}{d-2}$; these lines approach the line containing the $l^{2}(\Z^d)\times l^{2}(\Z^d) \to l^{1}(\Z^d)$ bounds as $d \to \infty$ so we approach the full Banach range of estimates as $d \to \infty$.  By trivially estimating the operator in $l^\infty(\Z^d)$, we also have that $T^*$ is bounded on $l^{\infty}(\Z^d)\times l^{\infty}(\Z^d) \to l^{\infty}(\Z^d)$.  Interpolating these three bounds and noting the nesting properties of the discrete $l^p$ spaces, that is: 
\[
\|f\|_{l^q(\Z^d)} \leq \|f\|_{l^p(\Z^d)} \text{ for all } 1 \leq p \leq q \leq \infty,
\]
leads to the Corollary below.
\end{remark}

\begin{cor}
$T^*$ is bounded $l^{p}(\Z^d)\times l^{q}(\Z^d) \to l^{r}(\Z^d)$ for all $d \geq 3$, $\frac{1}{p} + \frac{1}{q} \geq \frac{1}{r}$, $r>\frac{d}{d-2}$ and  $p,q\geq 1$.
\end{cor}

The key feature of this Corollary is the wide range of H\"older estimates obtained, while the broader estimates follow immediately from the nesting property of the discrete $l^p$ spaces. As mentioned it would be interesting to see what the full range of bounds (and most importantly the full H\"older range) for this operator are.


The paper is organized as follows: we begin with some necessary conditions for boundedness in Section \ref{necessary conditions section}.  In Section \ref{circle method section} we use the circle method to decompose our operator.  We handle the error from the minor arcs in Section \ref{minor arc section}, the rest of the error in Section \ref{majorarcerror}, and the main term from the decomposition in Section \ref{mainterm}, where we prove Theorem \ref{Main theorem}.  We comment on multilinear extensions in the final section. 

\subsection{Acknowledgements}
T. C. Anderson was supported in part by NSF DMS-1502464.  E. A. Palsson was supported in part by Simons Foundation Grant \#360560.

\section{Necessary conditions}
\label{necessary conditions section}
We begin by relating some necessary bounds for the bilinear (and multilinear) operators that we consider.  Note that if we know that an operator $T$ is bounded on $l^{p_0}(\Z^d)\times l^{q_0}(\Z^d) \to l^{r_0}(\Z^d)$, then we automatically get all bounds $l^{p}(\Z^d)\times l^{q}(\Z^d) \to l^{r}(\Z^d)$ for all $p \leq p_0, q \leq q_0, r \geq r_0$ due to the nestedness properties of the discrete norms.

\begin{lemma}
If $T^*(f_1, \dots , f_m)$ is bounded on $l^{p_1}(\Z^d)\times\cdots \times l^{p_m}(\Z^d) \to l^{r}(\Z^d)$, then  $\frac{1}{r} \leq \frac{1}{p_1}+\cdots +\frac{1}{p_m}$. 
\end{lemma}
Therefore for $T^*$ to be bounded on $l^{p}(\Z^d)\times l^{q}(\Z^d) \to l^{r}(\Z^d)$, we need $\frac{1}{r} \leq \frac{1}{p}+\frac{1}{q}$. (So in the bilinear case, the best $l^2(\Z^d)$ bounds we can expect are $l^2(\Z^d)\times l^2(\Z^d) \to l^1(\Z^d)$ bounds). 
 
\begin{proof}
We focus on the bilinear setting -- minor modifications yield the multilinear result.  We also focus on the case when $1 \leq p,q,r <\infty$.
We use a scaling argument:
Let $f = \chi_{[0,L)}$ and $g = \chi_{[0,L)}$.  Then we have that 
\[
\|T^*(f\otimes g)(\bm{y})\|_{l^r(\Z^d)} = (\sum_{\bm{y}\in\Z^d}(\sup_{\lambda} \lambda^{1-d} \sum_{\bm{u}^2+\bm{v}^2=\lambda}\chi_{[0,L)^d}(\bm{y}-\bm{u})\chi_{[0,L)^d}(\bm{y}-\bm{v}))^r)^{1/r}.
\]
For each $\bm{y}$, the inner expression will be nonzero only if $\bm{u}^2 + \bm{v}^2 = \lambda$ and $u_i \leq y_i<u_i+L$ as well as $v_i\leq y_i < v_i+L$ for all $1\leq i \leq d$.  Since $L$ is fixed, when $\lambda$ gets large, there are $L^d$ such $\bm{y}$ that contribute to the sum, giving
\[
(L^d(\sup_{\lambda} \lambda^{1-d}\#\{|\bm{u}|^2 + |\bm{v}|^2 = \lambda\}^r))^{1/r} \lesssim L^{d/r}
\]
since there are asymptotically $\lambda^{d - 1}$ such $(\bm{u},\bm{v})$.  On the other hand by an even simpler calculation
\[
\|f\|_{l^p(\Z^n)}\| g\|_{l^q(\Z^n)} = L^{d/p}L^{d/q}
\]
Hence to have $L^{d/r} \lesssim L^{d/p}L^{d/q}$ we must have $\frac{1}{r} \leq \frac{1}{p}+\frac{1}{q}$, or more generally for the $m$-linear variant, $\frac{1}{r} \leq \frac{1}{p_1}+\cdots +\frac{1}{p_m}$.

\end{proof}

\section{Set up and decomposition}
\label{circle method section}
We now turn to the proof of Theorem \ref{Main theorem}.  The first key point to note is that we can pull out the function $g$ in $l^\infty$ norm and reduce matters to considering $l^p(\Z^d)\to l^p(\Z^d)$ bounds for an operator $T_0$, see for example Barrionuevo et al \cite{BGHHO}; indeed we have
\begin{equation}
    T^*(f,g)(\bm{x}) \leq \|g\|_{l^\infty(\Z^d)}\cdot T^*_0(|f|)(\bm{x})
\end{equation}
where 
\begin{equation}
    T_0(f)(\bm{x}) :=  \frac{1}{N(\lambda)}\sum_{\bm{u}^2+\bm{v}^2=\lambda}f(\bm{x}-\bm{u}).
\end{equation}
and 
\begin{equation}
    T^*_0(f) :=  \sup_{\lambda >0}|T_0(f)|.
\end{equation}

Therefore we have that $T_0(|f|) = \left(|f|\otimes 1 \right)*\sigma_\lambda$, so 

\[\widehat{T_0(|f|)}(\bm{\xi}) = \left(\widehat{|f|}\otimes \delta_0 \right)(\bm{\xi})\cdot\hat{\sigma}_\lambda(\bm{\xi}) = \frac{1}{N(\lambda)}\sum_{\bm{u}^2+\bm{v}^2=\lambda}\widehat{|f|}(\bm{\xi})e(\bm{u}\cdot\bm{\xi})
\]
where $\bm{\xi}\in \T^d$ and $e(x) = e^{2\pi ix}$.

So we can rewrite
\[
\widehat{T_0(|f|)} = \widehat{|f|}\hat{\sigma}_{\lambda, 0}
\]
where 
\begin{equation}
    \hat{\sigma}_{\lambda, 0}(\bm{\xi}) = \frac{1}{N(\lambda)}\sum_{\bm{u}^2+\bm{v}^2=\lambda}e(\bm{u}\cdot\bm{\xi})
\end{equation}

We will start by using the circle method to decompose the Fourier transform of the arithmetic surface measure $\hat{\sigma}_{\lambda,0} (\xi)$.
The circle method will lead us to the following decomposition:
\begin{equation}
\label{maindecomposition}
   T_0 = M_\lambda + E_\lambda = M_\lambda + (A_\lambda - M_\lambda) + E_{m,\lambda} := I+II+III
\end{equation}
where the term I is the main term coming from the major arcs, term II is the major arc approximation error term, and term II is the error term coming from the minor arcs (we emphasize that $T_0$ depends on $\lambda$ even though we suppress this notation).  We will prove $l^{p}(\Z^d) \to l^p(\Z^d)$ bounds for the maximal operators arising from each of these terms.  The process begins in a similar manner to \cite{MSW} and \cite{Magyar_ergodic}.

\label{minorarcs}
Let $\Lambda \leq \lambda < 2\Lambda $ and call $N = \Lambda^{1/2}$.
Applying the circle method to $\hat{\sigma}_{\lambda,0}$, we get that 
\[
\hat{\sigma}_{\lambda,0} = \frac{1}{N(\lambda)}\sum_{0\leq u_i,v_i \leq N}e(\bm{u}\cdot\bm{\xi})\int_{\T}e(\theta(\bm{u}^2+\bm{v}^2 - \lambda))d\theta
\]
\[
 = \frac{1}{N(\lambda)}\int_{\T}\prod_{i=1}^d\sum_{u_i \leq N}e(\theta u_i^2+\xi_i u_i)\prod_{j=1}^d\sum_{v_j\leq N}e(\theta v_j^2)e(-\lambda\theta)d\theta
 := \int_{\T}\prod_{i=1}^d S_N(\theta, \xi_i)\prod_{j=1}^d S_N(\theta)e(-\lambda\theta)d\theta 
\]
\[
:= \frac{1}{N(\lambda)}\int_{\T}F(\theta, \bm{\xi})F(\theta)e(-\lambda\theta)d\theta. 
\]

We will decompose this Fourier transform as $\hat{\sigma}_{\lambda,0} = \hat{M}_\lambda+\hat{E}_\lambda$, where $\hat{M}_\lambda$ will come from the major arcs and $\hat{E}_\lambda$ will come from the minor arc piece as well as error from the major arc approximation.
Define the major arc centered at the rational $a/q$
\[
M_{a/q} := \{\theta\in\T : |\theta-\frac{a}{q}|\leq \frac{1}{8qN}\},
\]
and the major arcs
\[
M:= \bigcup_{1\leq q \leq N}\bigcup_{(a,q)=1, a\leq q}M_{a/q}
\]
and let $m = \T\setminus M$ be the minor arcs.

On the major arcs, let $\theta = \frac{a}{q}+\beta$ where $|\beta| \leq \frac{1}{8qN}$ and split $\bm{u} = q\bm{\tilde{u}}+\bm{y}, \bm{v} = q\bm{\tilde{v}}+ \bm{z}$ into residue classes.  Define a smooth compactly supported function $\Phi(\bm{x})$ such that $\Phi(\bm{x}) = 1$ for $|\bm{x}| <1$.  Then on $M_{a/q}$, we have
\[
\frac{1}{N(\lambda)}\int\sum_{\bm{y}\in \Z_q^d}\sum_{\bm{z}\in \Z_q^d}\sum_{\bm{\tilde{u}}\in \Z^d}\sum_{\bm{\tilde{v}}\in\Z^d}e((\frac{a}{q}+\beta)(q\bm{\tilde{u}}+\bm{y})^2+\bm{\xi}\cdot (q\bm{\tilde{u}}+\bm{y})+(\frac{a}{q}+\beta)(q\bm{\tilde{v}}+\bm{z})^2)\Phi_1(\frac{\bm{u}}{N})\Phi_2(\frac{\bm{v}}{N})e(-\lambda(\frac{a}{q}+\beta))d\beta
\]
Let $B(\bm{x}) = e(\beta \bm{x}^2)\Phi(\frac{\bm{x}}{N})$.  We apply Poisson summation to get
\[
\frac{1}{N(\lambda)}e(\frac{-\lambda a}{q})q^{-d}\sum_{\bm{y}\in \Z_q^d}\sum_{\bm{l}\in \Z^d}e(\frac{a\bm{y}^2}{q})e(\frac{\bm{l}\cdot \bm{y}}{q})q^{-d}\sum_{\bm{z}\in \Z_q^d}\sum_{\bm{m}\in\Z^d}e(\frac{\bm{m}\cdot \bm{z}}{q})\int_{|\beta| <\frac{1}{8Nq}}e(-\lambda\beta)\hat{B}(\bm{\xi}-\frac{\bm{l}}{q})\hat{B}(-\frac{\bm{m}}{q})d\beta.
\]
We define the Gauss sum $G(\bm{l},a,q) = q^{-d}\sum_{\bm{y}\in \Z_q^d}e(\frac{a\bm{y}^2}{q})e(\frac{\bm{l}\cdot \bm{y}}{q})$ and the sum $G(\bm{m},0,q) = q^{-d}\sum_{\bm{z}\in \Z_q^d}e(\frac{\bm{m}\cdot \bm{z}}{q})$.  Note that the sum $G(\bm{m},0,q)$, despite our notation, is not a Gauss sum since there is no quadratic term.   Due to the exponential integral piece $\hat{B}(-\frac{\bm{m}}{q})$, we cannot use orthogonality, and with this in mind, the above equals
\begin{equation}
\label{Aoperator}
\hat{A}^{a/q}_\lambda :=\frac{1}{N(\lambda)}e(\frac{-\lambda a}{q})\sum_{\bm{l}\in \Z^d}G(\bm{l},a,q)\sum_{\bm{m}\in\Z^d}G(\bm{m},0,q)\int\limits_{|\beta|\leq 1/8qN} e(-\lambda\beta)\hat{B}(\bm{\xi}-\frac{\bm{l}}{q})\hat{B}(-\frac{\bm{m}}{q})d\beta.
\end{equation}
We insert smooth cutoff functions $\Psi_1$, $\Psi_2$ (where $\Psi_1(\bm{\xi}) = 1$ for $|\bm{\xi}| < 1$ and similarly for $\Psi_2$) to define the approximate multiplier
\begin{equation}
\label{Boperator}
\hat{B}^{a/q}_\lambda :=
\frac{1}{N(\lambda)}e(\frac{-\lambda a}{q})\sum_{\bm{l}\in \Z^d}G(\bm{l},a,q)\sum_{\bm{m}\in\Z^d}G(\bm{m},0,q)\Psi_1(q\bm{\xi}-\bm{l})\Psi_2(-\bm{m})\int\limits_{|\beta|\leq 1/8qN} e(-\lambda\beta)\hat{B}(\bm{\xi}-\frac{\bm{l}}{q})\hat{B}(-\frac{\bm{m}}{q})d\beta.
\end{equation}
Note that only the $\bm{m}=0$ term contributes to $\Psi_2$.  One may wonder why $\Psi_2$ was inserted, since it always localizes to the zero frequency.  The reason for inserting such a localization is important for the main term analysis and we comment on this then.

Now extend the integration to the whole real line to define the approximate multiplier:
\begin{equation}
\label{Coperator}
\hat{C}^{a/q}_\lambda := \frac{1}{N(\lambda)}e(\frac{-\lambda a}{q})\sum_{\bm{l}\in \Z^d}G(\bm{l},a,q)\sum_{\bm{m}\in\Z^d}G(\bm{m},0,q)\Psi_1(q\bm{\xi}-\bm{l})\Psi_2(-\bm{m})\int_\R e(-\lambda\beta)\hat{B}(\bm{\xi}-\frac{\bm{l}}{q})\hat{B}(-\frac{\bm{m}}{q})d\beta.
\end{equation}
Now we can identify, as in \cite{Stein93} (note that we have replaced sharp cutoffs with smooth ones) the exponential integral in $\beta$ with
\[
\RFT{d\sigma}_{\lambda^{1/2}}((\bm{\xi}\otimes \bm{0})-(\frac{\bm{l}}{q}\otimes\frac{\bm{m}}{q}))
\]
which is the continuous spherical surface measure on the sphere of radius $\lambda^{1/2}$ in $\R^{2d}$, so $\hat{C}^{a/q}_\lambda = \hat{M}^{a/q}_\lambda$.  Note that this symbol enjoys the Fourier decay
\begin{equation}
\label{Fourier decay}
    \RFT{d\sigma}((\bm{\xi}\otimes\bm{\eta})) \lesssim (1+|\bm{\xi}|+|\bm{\eta}|)^{-\frac{2d-1}{2}}.
\end{equation}
Summing over $q,a$, we have that 
\[
\hat{M}_\lambda = \sum_{q=1}^N\sum_{a\in\Z_q}e(\frac{-\lambda a}{q})\sum_{\bm{l}\in \Z^d}G(\bm{l},a,q)\sum_{\bm{m}\in\Z^d}G(\bm{m},0,q)\Psi_1(q\bm{\xi}-\bm{l})\Psi_2(-\bm{m})\RFT{d\sigma}_{\lambda^{1/2}}((\bm{\xi}\otimes \bm{0})-(\frac{\bm{l}}{q}\otimes\frac{\bm{m}}{q})).
\]

\section{Minor arcs}
\label{minor arc section}
Here we show $l^p(\Z^d)\to l^p(\Z^d)$ bounds for the minor arc multiplier
\[
\hat{E}_{m,\lambda} = \frac{1}{N(\lambda)}\int_{m}F(\theta, \bm{\xi})F(\theta)e(-\lambda\theta)d\theta. 
\]
This approach follows \cite{ACHK} with minor changes.  We sketch the details.

We proceed by showing an $l^2(\Z^d)\to l^2(\Z^d)$ bound for a dyadic version of the operator $E_{m,\lambda}$, with some power decay in $N$, that is
\begin{equation}
\label{dyadicerrordecay}
\|\sup_{\lambda\in [\Lambda, 2\Lambda)}|E_{m,\lambda} |\|_{l^2(\Z^d)\to l^2(\Z^d)} \lesssim N^{-\delta}
\end{equation}
for some $\delta >0.$

First, we adopt a proposition from \cite{ACHK}.  Its proof is very similar, but we include a brief sketch for completion.
\begin{prop}
$\|\sup_{\lambda\in [\Lambda, 2\Lambda)}|E_{m,\lambda} |\|_{l^2(\Z^d)\to l^2(\Z^d)} \lesssim \frac{1}{N(\Lambda)}\int_m\sup_{\bm{\xi}\in\T^d}|F(\theta, \bm{\xi})||F(\theta)|d\theta$
\end{prop}
\begin{proof}
First note that
\[
|E_{m,\lambda} (f)(\bm{u})| \leq \frac{1}{N(\Lambda)}\int_m|\int_{\T^d}F(\theta,\bm{\xi})F(\theta)\hat{f}(\bm{\xi})e(-\bm{u}\cdot\bm{\xi})d\xi |d\theta
\]
and call $h(\theta, \bm{u}) := \int_{\T^d}F(\theta,\bm{\xi})F(\theta))\hat{f}(\bm{\xi})e(-\bm{u}\cdot\bm{\xi})d\bm{\xi}$.
Now we have
\[
\|\sup_{\lambda\in [\Lambda, 2\Lambda)}|E_{m,\lambda} (f) |\|_{l^2(\Z^d)}
\leq \frac{1}{N(\Lambda)}\|\int_m|h(\theta,\bm{u})d\theta\|_{l^2(\Z^d)}
\leq \frac{1}{N(\Lambda)}\int_m(\sum_{\bm{u}\in\Z^d}|h(\theta, \bm{u})|^2)^{1/2}d\theta
\]
using Minkowski's integral inequality.  After an application of Bessel's inequality the above is bounded by
\[
\frac{1}{N(\Lambda)}\int_m(\int_{\T^d}F(\theta,\bm{\xi})F(\theta))\hat{f}(\bm{\xi})|^2d\bm{\xi})^{1/2}d\theta
\]
\[
\leq \frac{1}{N(\Lambda)}\int_m\sup_{\bm{\xi}}|F(\theta, \bm{\xi})||F(\theta)|d\theta(\int_{\T^d}|\hat{f}(\bm{\xi})|^2d\bm{\xi})^{1/2}
\]
\[
\leq \frac{1}{N(\Lambda)}\|\hat{f}\|_{L^2(\T^d)}\int_m\sup_{\bm{\xi}}|F(\theta, \bm{\xi})||F(\theta)|d\theta
\]
and after applying Plancherel we get
\[
\|f\|_{l^2(\Z^d)}\frac{1}{N(\Lambda)}\int_m\sup_{\bm{\xi}}|F(\theta, \bm{\xi})||F(\theta)|d\theta.
\]
\end{proof}

Next using the classic Weyl's inequality (see \cite{Vaughan}), we get $\sup_{\bm{\xi}} |S_N(\theta, \bm{\xi})| \lesssim N^{1/2+\varepsilon}$, so we therefore have (for any $\varepsilon >0$),
\[
\int_m\sup_{\bm{\xi}\in\T^d}|F(\theta, \bm{\xi})||F(\theta)|d\theta \lesssim 
N^{d+\varepsilon} = N^{2d-2-(d-2)+\varepsilon}
\]
which is \eqref{dyadicerrordecay} for $\delta = d-2-\varepsilon$.

We also have that 
\begin{equation}
\label{dyadicl1error}
    \|\sup_{\lambda\in [\Lambda,2\Lambda)}|E_{m,\lambda} |\|_{l^1(\Z^d)\to l^1(\Z^d)} \lesssim N^2
\end{equation}
Since 
\[
\|E_{m,\lambda}(f)(\bm{u})\|_{l^1(\Z^d)} \leq \frac{1}{N(\lambda)}\int_m|\int_{\T^d}F(\theta,\bm{\xi})F(\theta))\hat{f}(\bm{\xi})e(-\bm{u}\cdot\bm{\xi})d\bm{\xi}|d\theta 
\]
\[
\leq \frac{N^{2d}}{N(\Lambda)}\|\int_m|f(\bm{u})\|_{l^1(Z^d)}\leq N^2|m|\|f\|_1.
\]

Hence we can interpolate the gain from \eqref{dyadicerrordecay} with \eqref{dyadicl1error} to get:
\[\|\sup_{\lambda\in [\Lambda,2\Lambda)}|E_{m,\lambda} |\|_{l^p(\Z^d)\to l^{p}(\Z^d)} \lesssim N^{\alpha_p}
\]
where $\alpha_p = 2(2/p-1)-\delta (2-2/p)$.
  If $p>\frac{2+\delta}{1+\delta}$ then we have that $\alpha_p<0$.  We can take any $0 < \delta < d-2$, so taking $\delta$ as close to $d-2$ as we wish, we get this bound for $p>\frac{d}{d-1}$.
 Then we sum up over dyadic ranges to get

 \[
\left\| \sup_{\lambda} \big|E_{m,\lambda} | \big| \right\|_{l^p(\mathbb Z^d) \to l^p(\mathbb Z^d)} \leq \sum_{N = 2^j}\|\sup_{\lambda\in [\Lambda,2\Lambda)}|E_{m,\lambda} |\|_{l^p(\Z^d)\to l^p(\Z^d)} 
\]
\[
\lesssim \sum_{N= 2^j}N^{\alpha_p} \lesssim \sum_j2^{\alpha_p j} \leq C.
\] 
which yields $l^p(\mathbb Z^d) \to l^p(\mathbb Z^d)$ bounds for $E_{m,\lambda}$ for all $p >\frac{d}{d-1}$.

\section{Major arc error terms}
\label{majorarcerror}

Here we will show that the error incurred from using the operator $M_\lambda$ instead of $A_\lambda$ is small, namely:
\begin{equation}
\label{approximationerrorlp}    
\left\| \sup_{\lambda \in [\Lambda/2, \Lambda)} \big| A_\lambda-M_\lambda \big| \right\|_{l^p(\mathbb Z^d) \to l^p(\mathbb Z^d)} \lesssim \Lambda^{-\gamma_p}.
\end{equation}
for some $\gamma_p >0$.  This is standard, but for completeness we quickly sketch the details.
The estimate \eqref{approximationerrorlp}
\[
\left\| \sup_{\lambda} \big| A_\lambda-M_\lambda \big| \right\|_{l^p(\mathbb Z^d) \to l^p(\mathbb Z^d)} \leq \sum_{\lambda\approxeq 2^j}\left\| \sup_{\lambda \in [\Lambda/2, \Lambda)} \big| A_\lambda-M_\lambda \big| \right\|_{l^p(\mathbb Z^d) \to l^p(\mathbb Z^d)}
\]
\[
\lesssim \sum_j2^{-\gamma_p j} \leq C.
\]

To show \eqref{approximationerrorlp}, we will show
\begin{equation}
\label{approximationerrorl2}
\left\| \sup_{\lambda \in [\Lambda/2, \Lambda)} \big| A_\lambda-M_\lambda \big| \right\|_{l^2(\mathbb Z^d) \to l^2(\mathbb Z^d)}\lesssim \Lambda^{-\beta_p}
\end{equation}
which is Proposition 4 of \cite{Magyar_ergodic}; see also \cite{Hughes_Vinogradov}, and interpolate this with the estimate
\begin{equation}
\label{approximationerrortrivial}
    \left\| \sup_{\lambda \in [\Lambda/2, \Lambda)} \big| A_\lambda-M_\lambda \big| \right\|_{l^p(\mathbb Z^d) \to l^p(\mathbb Z^d)} \lesssim 1.
\end{equation}
Given any $\varepsilon >0$, this will prove \eqref{approximationerrorlp} for any $\frac{d}{d-2}+\varepsilon < p \leq 2$ as long as $\beta_p >0$.
To prove \eqref{approximationerrortrivial}, we simply combine the estimates \[
\left\| \sup_{\lambda \in [\Lambda/2, \Lambda)} \big| A_\lambda \big| \right\|_{l^p(\mathbb Z^d) \to l^p(\mathbb Z^d)}
\lesssim 1
\]
and 
\[
\left\| \sup_{\lambda \in [\Lambda/2, \Lambda)}\big| M_\lambda \big| \right\|_{l^p(\mathbb Z^d)\to l^p(\mathbb Z^d)}
\lesssim 1.
\]
The latter estimate is true for $\frac{d}{d-2} <p \leq 2$ due to Section \ref{mainterm}, and the former is true by \cite{Magyar_dyadic}.


\section{Main Term}
\label{mainterm}
Here we estimate the $l^p\to l^p$ norm of the main term.  Firstly, using the triangle inequality,
\[
    \|\sup_\lambda |M_\lambda|\|_p \leq \sum_{q=1}^\infty \sum_{a\in U_q}\|\sup_\lambda |M_\lambda^{a,q}|\|_p\]
where we recall the multiplier $\hat{M}_\lambda^{a,q}$ defined in \eqref{Coperator}:
\[
\hat{M}_\lambda^{a,q} = 
\sum_{\bm{l}\in \Z^d}G(\bm{l},a,q)\sum_{\bm{m}\in\Z^d}G(\bm{m},0,q)\Psi_1(q\bm{\xi}-l)\Psi_2(-\bm{m})\RFT{d\sigma}_{\lambda^{1/2}}((\bm{\xi}\otimes \bm{0})-(\frac{\bm{l}}{q}\otimes\frac{\bm{m}}{q}).
\]
The multiplier $\hat{M}_\lambda^{a,q}$ naturally splits up into the product of two multipliers as in \cite{MSW}
\[
\hat{S}(\bm{\xi})=\sum_{\bm{l}\in \Z^d}\sum_{\bm{m}\in\Z^d}G(\bm{l},a,q)G(\bm{m},0,q)\Psi_1(q\bm{\xi}-l)\Psi_2(-\bm{m})
\]
    and \[\hat{S}'_\lambda(\bm{\xi}) = \sum_{\bm{l}\in \Z^d}\sum_{\bm{m}\in \Z^d}\Psi_1'(q\bm{\xi}-l)\Psi_2'(-\bm{m})\RFT{d\sigma}_{\lambda^{1/2}}((\bm{\xi}\otimes \bm{0})-(\frac{\bm{l}}{q}\otimes\frac{\bm{m}}{q}))
\]
where $\Psi_1'$ is an appropriate cutoff functions such that $\Psi_1'\Psi_1 = \Psi_1$, and similarly for $\Psi_2'$.  At this stage it is important to have the term $\Psi_2$ present -- without this localizing term (which reduces the sum in $\bm{m}$ to a single term), we could not split the multiplier in this way, since this splitting relies on the fact that for each $\bm{\xi}$, there is only one $\bm{l}$ and one $\bm{m}$ that contribute to the sum.  


Since $M_\lambda^{a,q} = S'_\lambda\circ S = S\circ S'_\lambda$, It suffices to bound both $S$ and $\sup_\lambda |S'_\lambda|$ in $l^p(\Z^d)$. 

To bound $S'$, we use the bounds for the continuous version of the bilinear spherical maximal function from \cite{JL}.  Note that due to the extra Fourier decay, we are able to get $l^p(\Z^d)$ bounds for all $p>1$.
\begin{prop}
$\|\sup_\lambda |S'_\lambda|\|_{l^p(\Z^d)\to l^p(\Z^d)} \leq C$ for all $d \geq 2, p>1$
\end{prop}
\begin{proof}
Fist note that we have $\hat{S}'_\lambda(\bm{\xi}) = \hat{U}_\lambda(\bm{\xi},\bm{0})$ where
\begin{equation}
    \hat{U}_\lambda(\bm{\xi},\bm{\eta}) = \sum_{\bm{l}\in \Z^d}\sum_{\bm{m}\in \Z^d}\Psi_1'(q\bm{\xi}-\bm{l})\Psi_2(q\bm{\eta}-\bm{m})\RFT{d\sigma}_{\lambda^{1/2}}((\bm{\xi}\otimes \bm{\eta})-(\frac{\bm{l}}{q}\otimes\frac{\bm{m}}{q})).
\end{equation}
This is now a symbol in $\T^{2d}$.  We can now apply Magyar-Stein-Wainger transference \cite{MSW} to $U$, followed by an application of the boundedness of the bilinear spherical maximal function in \cite{JL} to get
\[
\|\sup_\lambda |U_\lambda|\|_{l^p(\Z^{2d})} \leq \|\sup_\lambda |U_\lambda|\|_{L^p(\R^{2d})} \leq C,
\]
where we have used the decay of our symbol in \eqref{Fourier decay} to compare it to mollified bilinear spherical averages in \cite{JL} via the method of Rubio de Francia \cite{RdF}.

Finally, we have that since $\hat{S}'_\lambda(\bm{\xi}) = \hat{U}_\lambda(\bm{\xi},\bm{0})$, then
\[
\|\sup_\lambda |S'_\lambda(f)|\|_{l^p(\Z^{d})} = \|\sup_\lambda |U_\lambda(f,\delta_0)|\|_{l^p(\Z^{2d})}
\]
and since $\|\delta_0\|_{l^p(\Z^d)} = 1$,
\[
\|\sup_\lambda |U_\lambda(f,\delta_0)|\|_{l^p(\Z^{2d})} \leq \|\sup_\lambda |U_\lambda|\|_{l^p(\Z^{2d})}
\]
which finishes the proof.

\end{proof}


To bound the operator $S$ notice that now $\sum_{\bm{m}\in \Z_q^d}G(\bm{m},0,q)\Psi_2(-\bm{m})$ = $G(\bm{0},0,q) = 1$ so that our multiplier now takes the form of those considered in \cite{MSW}.  We have:
\begin{prop}
$\|S\|_{l^p(\Z^d)\to l^p(\Z^d)} \leq q^{-d(1-1/p)+\varepsilon}$
\end{prop}
\begin{proof}
Using Proposition 2.2 of \cite{MSW} with  $\gamma = G(\bm{l},a,q)$ we have that
\[
\|S\|_{l^p(\Z^d)\to l^p(\Z^d)} \leq (\sup_{\bm{l}}|G(\bm{l},a,q)|)^{2-2/p}.
\]
  Recalling $\sup_{\bm{l}}|G(\bm{l},a,q)| \leq q^{-d/2}$, we get the desired bound.
\end{proof}

Summing over $q$ and $a$, we get
\[
\|\sup_\lambda |M_\lambda| \|_{l^p(\Z^d)} \leq \sum_{q=1}^\infty\sum_{a\in U_q}q^{-d(1-1/p)} <\infty
\]
if and only if $p>\frac{d}{d-2}$.  Therefore, the arithmetic term provides the bottleneck for boundedness, with the restriction $p> \frac{d}{d-2}$, which matches the bounds in the linear setting.  Note that unlike the linear setting, we can take $d \geq 3$ instead of $d \geq 5$.

\subsection{Proof of Theorem \ref{Main theorem}}\label{proof of main theorem}
We now complete the proof of Theorem \ref{Main theorem}.  Combining the restriction on $p$ from the error estimates along with the sufficient conditions for the main term, we see that the full operator $T^*_0$ is bounded on $l^p(\Z^d)\to l^p(\Z^d)$ for all $d\geq 3$, $p > \frac{d}{d-2}$.  Therefore $T^*$ is bounded on $l^p(\Z^d)\times l^\infty(\Z^d)\to l^p(\Z^d)$ for all $d \geq 3$, $p> \frac{d}{d-2}$.

By taking the example $f = \delta_0$, $g \equiv 1$, one can see that for $\lambda = |\bm{x}|^2$ that $\|T_\lambda(f,g)\|^p_{l^p(\Z^d)} \geq \sum_{\bm{x}\in \Z^d}\big(\frac{1}{|\bm{x}|^{2(d-1)}}\big)^p$, which converges if and only if $p>\frac{d}{2(d-1)}$, therefore for $p\geq 1$.  Similarly, one can take $\lambda = n|\bm{x}|^2$ for any natural number $n$, to reduce matters to estimating 
\[\sum_{\bm{x}\in \Z^d}\big(\frac{1}{|\bm{x}|^{2(d-1)}}\# \{ \bm{v} : |\bm{v}|^2 = \frac{n-1}{n}\lambda \}\big)^p
\]
The count in this sum is $\approxeq |\bm{x}|^{d-2}$ for $d \geq 5$ by the Hardy-Littlewood asymptotic, so we get 
\[
\sum_{\bm{x}\in \Z^d}\big(|\bm{x}|^{-d}\big)^p.
\]  For $d=4$, we use the fact that for $(\frac{n-1}{n})\lambda = 1 \mod{8}$, we have the same asymptotic, and for $d=3$, use the fact that for infinitely many even $\lambda$, we have that $\# \{ \bm{v} : |\bm{v}|^2 = \frac{n-1}{n}\lambda \}$ is nonzero.  The most restrictive of these estimates yield $p>1$.  It would be interesting to see if $p>1$ is also the sharp range of boundedness.

\section{Multilinear results}
We now mention the $l$-linear results that we obtain, which comes by interpolation with an $l^\infty(\Z^d) \times \dots \times l^\infty(\Z^d) \times l^p(\Z^d)\to l^p(\Z^d)$ bound for $T^*$.  Our proofs carry through in this setting; we only indicate the necessary changes.  Firstly, the count $N(\lambda)$ is approximately $\lambda^{\frac{ld}{2}-1}$ by the Hardy-Littlewood asymptotic as long as $d> 4/l$.  Secondly, in the error term analysis, we get the dyadic $l^2(\Z^d)$ bound of $N^{(ld-2)-(\frac{ld}{2}-2-\epsilon)} = N^{ld-2-\delta}$ for $\delta = \frac{ld}{2}-2-\epsilon$.  Interpolating with the trivial $l^1(\Z^d)$ estimate of $N^2$, we get $l^p$ bounds for all $p > \frac{ld}{ld-2}$.  For the main term, we get the restriction $p> min \{ p_c, p_d\}$ where $p_c$ is the infimum of $p$ such that the operator $S'_\lambda$ is bounded (which stems from continuous bounds for the multilinear spherical maximal function), and $p_d = \frac{d}{d-2}$ is still the infimum of all $p$ such that $S$ is bounded. So we have that the $l$-linear variant is bounded on $l^\infty(\Z^d) \times \dots \times l^\infty(\Z^d) \times l^p(\Z^d)\to l^p(\Z^d)$ for all $p>\frac{d}{d-2}$. Through interpolation this leads to bounds $l^{p_1}(\Z^d)\times\ldots\times l^{p_l}(\Z^d) \to l^{r}(\Z^d)$ for $\frac{1}{p_1} + \ldots + \frac{1}{p_l} \geq \frac{1}{r}$, $r>\frac{d}{d-2}$ and  $p_1,\ldots,p_l\geq 1$.

\bibliographystyle{amsplain}

\end{document}